\documentclass{article}

\usepackage[12pt]{extsizes}
\usepackage{cmap}
\usepackage[utf8]{inputenc}
\usepackage[T2A]{fontenc}
\usepackage[english]{babel}
\usepackage{amsmath}
\usepackage{amsthm}
\usepackage{amssymb}
\usepackage{amsfonts}
\usepackage{euscript}
\usepackage{enumerate}
\usepackage{hyperref}
\usepackage[affil-it]{authblk}

\usepackage[all]{xy}
\usepackage{epigraph}

\newcommand{\RN}{\mathbb{R}} 
\newcommand{\R}{\mathbb{R}}
\renewcommand{\C}{\mathbb{C}} 
\newcommand{\CN}{\mathbb{C}}
\newcommand{\ZN}{\mathbb{Z}} 
\newcommand{\Z}{\mathbb{Z}}

\newcommand{\eps}{\ensuremath\varepsilon}

\newcommand{\prodl}{\prod\limits}

\newcommand{\ovl}{\overline}

\newcommand{\gr}{\mathop{\mathrm{gr}}\nolimits}

\newcommand{\Ker}{\mathop{\mathrm{Ker}}\nolimits}
\renewcommand{\Im}{\mathop{\mathrm{Im}}\nolimits}
\newcommand{\Hom}{\mathop{\mathrm{Hom}}\nolimits}

\newcommand{\GL}{\mathop{\mathrm{GL}}\nolimits}

\newcommand{\mf}{\mathfrak}

\newcommand{\mc}[1]{\mathcal{#1}}

\renewcommand{\Re}{\operatorname{Re}}

\newcommand{\abs}[1]{\lvert #1\rvert}

\newcommand{\rmi}{\mathrm{i}}
\DeclareMathOperator{\sgn}{sgn}
\newcommand{\A}{\mathcal{A}}
\begin{document}
\newtheorem{thr}{Theorem}[section]
\newtheorem*{thr*}{Theorem}
\newtheorem{lem}[thr]{Lemma}
\newtheorem*{lem*}{Lemma}
\newtheorem{cor}[thr]{Corollary}
\newtheorem*{cor*}{Corollary}
\newtheorem{prop}[thr]{Proposition}
\newtheorem*{prop*}{Proposition}
\newtheorem{stat}[thr]{Statement}
\newtheorem*{stat*}{Statement}
\newtheorem{example}[thr]{Example}

\theoremstyle{definition}
\newtheorem{defn}[thr]{Definition}
\newtheorem*{defn*}{Definition}
\theoremstyle{remark}
\newtheorem{rem}[thr]{Remark}
\newtheorem*{rem*}{Remark}
\numberwithin{equation}{section}
\author{Daniil Klyuev}
\title{Positive traces on quantized abelian Coulomb branches}
\maketitle
\begin{abstract}
Let $\mc{A}=\mc{A}_{G,N}^{\hbar=1}$ be a quantized Coulomb branch with an antilinear automorphism $\rho$. A map $T\colon \mc{A}\to\mathbb{C}$ is called a positive trace if $T(a\rho(a))>0$ for all nonzero $a\in \mc{A}$. Positive traces on Coulomb branches appear in the study of supersymmetric gauge theories. We classify positive traces on all abelian Coulomb branches, meaning $G=(\C^{\times})^d$ is a torus.
\end{abstract}
\section{Introduction}
Let $\mc{A}$ be a noncommutative algebra over $\C$ and $g$ be an automorphism of $\mc{A}$. A linear map $T\colon \mc{A}\to\C$ is called {\it $g$-twisted trace} if $T(ab)=T(bg(a))$ for all $a,b\in \mc{A}$. Let $\rho$ be an antilinear automorphism of $\mc{A}$. A $\rho^2$-twisted trace $T$ is called {\it positive} if $T(a\rho(a))>0$ for all nonzero $a\in \mc{A}$. 

To any reductive group $G$ and its representation $N$ one can attach a supersymmetric gauge theory. There is a certain variety $\mc{M}_{C}(G,N)$ associated to this theory, called Coulomb branch. Mathematical definition of Coulomb branches was given by Braverman, Finkelber and Nakajima in~\cite{BFN}. Their construction also provides a filtered deformation of $\C[\mc{M}_{C}(G,N)]$, we denote it by $\mc{A}_{G,N}$ and call quantized Coulomb branch algebra. The algebra $\mc{A}=\mc{A}_{G,N}$ also depends on certain parameters, called flavors. 

Positive traces are connected to irreducible spherical unitary representations of complex semisimple Lie groups as follows. Let $\mf{g}$ be a complex semisimple Lie algebra, $G$ be the corresponding simply connected Lie group. Let $K$ be a maximal compact subgroup of $G$. Let $V$ be an irreducible spherical unitary representation of $G$ considered as a real group. Spherical means that there exist nonzero $v\in V$ fixed by $K$. Consider the corresponding $(\mf{g},K)$-module $M=V^{K-fin}$. The action of $\mf{g}$ on $M$ is linear over real numbers. Consider the action of a (further) complexified $\mf{g}\otimes_{R}\C\cong \mf{g}\oplus \mf{g}$. It can be shown that we get a $U(\mf{g})$-bimodule such that the adjoint action of $\mf{g}$ is locally finite. Because $V$ is spherical and irreducible, $M$ is isomorphic to $U(\mf{g})/I$ with bimodule structure given by left and right multiplication. Here $I$ is a two-sided ideal that contains an ideal corresponding to some central character $\chi$. Let $\mc{A}=U(\mf{g})/I$ be the corresponding algebra. Then $T(a)=(a,1)$ is a positive trace for a certain choice of $\rho$. On the other hand, if $\mc{A}$ with a certain choice of $\rho$ admits a positive trace, then the completion of $\mc{A}$ with respect to $(a,b)=T(a\rho(b))$ will be an irreducible spherical unitary representation of $G$. 

In the case when $\mc{A}$ is a central reduction of $U(\mf{sl_n})$, it can be constructed as a Coulomb branch. Consider the linear quiver with dimension vector $(1,2,\ldots,n-1)$ and framing vector $(0,\ldots,0,n)$. Then taking $N$ to be representation space of the quiver and $G=\prod_{i=1}^{n-1}\GL_i$ with certain flavor parameters will give $\mc{A}_{G,N}$ isomorphic to $\mc{A}$. A similar construction is known for orthogonal and symplectic groups. 

From theoretical physics side, one should take certain $\rho$ and a positive trace $T$ ({\it sphere} trace) on $\mc{A}$~\cite{GaiottoSphere, GaiottoTeschner}.
Then completion of $\mc{A}$ with respect to $(a,b)=T(a\rho(b))$ gives a Hilbert space $\mc{H}$ with an action of, conjecturally, normal operators $L_a$, $R_a$ satisfying $L_a^*=R_{\rho(a)}$ and a spherical vector $v_0$ satisfying $L_av_0=R_av_0$. This information gives more precise information about the underlying gauge theory compared to just the algebra $\mc{A}$ itself. 

Positive traces are connected to star-products~\cite{BPR}. Twisted and positive traces were also considered in~\cite{DG, DFPY, DPY}. We classified positive traces with Etingof, Rains and Styker~\cite{EKRS} in the case of $G=\C^*$.

We consider the case when $\mc{A}$ is a quantized abelian Coulomb branch, this means that $G=T=(\C^*)^d$ is a torus. These algebras are also called hypertoric enveloping algebras. Previously the case of abelian Coulomb branches was considered from Higgs side in~\cite{GWTraces}. In Theorem 3.8 the authors express the sphere trace as an integral over a shift of $\mf{t}_{\R}$. This theorem was an important motivation for our paper. Furthermore, the authors express sphere trace as a sum of Verma traces and prove a certain symplectic duality statement. In the next version of our paper we plan to add sections on Verma-like traces for non-generic parameters and symplectic duality. 

Out motivation for this particular case $G=(\C^*)^d$ is as follows. Understanding twisted and positive traces on abelian Coulomb branches is an important step in classification of positive traces on general Coulomb branches. Any Coulomb branch can be included into a localized abelian Coulomb branch using equivariant localization map in homology (this is called ``abelianization map'' in physics). Denote this map $\phi\colon \mc{A}_{G,N}\to \mc{A}_{T,N,loc}$. In the upcoming article~\cite{KResidues}, we will describe the image of $\phi$. This allows us to write traces more explicitly in terms of $\phi(\mc{A}_{G,N})$ similarly to~\cite{GaiottoSphere, GaiottoTeschner}. There is an example of how this strategy works for the closely connected object in Chapter~2 of~\cite{KV}.

The plan of the article is as follows. In Subsection~\ref{SubSecAbelianCoulombBranches} we recall the results of chapter 4 of~\cite{BFN}: we write the basis of a quantized abelian Coulomb branches and the multiplication in terms of basis elements. In this case, the main feature of algebras from~\cite{EKRS} remains: there is a large commutative subalgebra that is a polynomial algebra $\C[\mf{t}]$ with generations in $\mc{A}_{\leq 2}$. This means that $\mc{A}$ is a graded by coweight lattice $Y$ and this grading is internal. In this case, a twisted trace is defined by its restriction to $\C[\mf{t}]$. In Subsection~\ref{SubSecTraceCondition} we write a condition for the map $T\colon\C[\mf{t}]\to \C$ to define a trace. In the unimodular case, defined below, we prove that for generic twist $g$ the number of trace is at most volume of {\it fundamental polytope} $\mc{P}=\sum [0,1]\xi_i$. 

In Subection~\ref{SubSecAnalyticFormulas} we deal with the case when all flavor parameters have real part in $(-\tfrac12,\tfrac12)$. We do it by writing trace in the form $T(R)=\int_{\mf{t}_{\R}}R(\rmi z)w(z)dz$, $R\in\C[\mf{t}]$. Here, similarly to~\cite{EKRS}, the function $w$ is a meromorphic function that exponentially decays at infinity and satisfies quasi-periodicity condition $w(z+\lambda)=e^{2\pi\rmi\zeta(\lambda)}w(z)$, where $\zeta$ is the parameter defining $g$ (FI parameter in physics terminology). In the case of $\Re b_i\in (-\tfrac12,\tfrac12)$ and generic $\zeta$ the dimension of the space of possible $w$ equals to the number of integer points in $\mc{P}+\zeta_{\Re}$, which in turn equal to the volume of $\mc{P}$. Therefore, in this case (and assuming unimodularity), all traces can be written as an integral. We also rewrite the positivity condition in terms of $w$:
\begin{prop*}[\ref{PropPositivityOnTermsOfWeight}]
Let $T$ be the trace constructed in Proposition~\ref{PropTraceAsAnIntegral} with weight $w(x)=e^{2\pi\zeta}\frac{Q(e^{2\pi z_1},\ldots, e^{2\pi z_d})}{\prod_{j=1}^n (e^{2\pi \xi_j}+e^{2\pi \rmi b_j})}$. Then $T$ is positive if and only if $Q\neq 0$ is a  polynomial that has nonnegative values on $\R^d$. 
\end{prop*}

After that, we use a strategy different to~\cite{EKRS}. In~\cite{EKRS} we have written all twisted traces explicitly and then for each trace we have given a condition for its positivity. In our case, because of complicated combinatorics of hyperplane arrangements, it is hard to write all traces explicitly, especially in non-unimodular case, so we prove that any positive trace must be of the form $\sum_{L\subset \mf{t}_{\R}} \int_L R(iz)w_L(z)dz$, where the sum is over some number of subspaces of $\mf{t}_{\R}$ and $w_L$ are some meromorphic exponentially decaying quasi-periodic functions as above. 

In order to do this we consider an expression that generalizes Taylor series of a Fourier transform $\mc{F}w$ at zero. This expression is given by
\[u_T(y)=\sum_{j_1,\ldots,j_d\geq 0}\frac{y_1^{j_1}\cdots y_d^{j_d}}{j_1!\cdots j_d!}T(e_1^{j_1}\cdots e_d^{j_d}),\] where $e_1,\ldots,e_d$ is a basis of $\mf{t}^*$. We describe basic properties of $u_T$ in Subsection~\ref{SubSecExpDefinition}.

In Subsection~\ref{SubSecMeromorphic} we prove prove the following theorem:
\begin{thr*}[\ref{ThrFourierTransformIsMeromorphic}]
\begin{equation*}
u_T=\frac{\sum_{\mu} S_{\mu}e^{\mu}}{\prod_{l\in I}(e^{\tfrac12\lambda_l}-e^{2\pi\rmi \zeta(\lambda_l)}e^{-\tfrac12\lambda_l})},
\end{equation*} where $S_{\mu}$ is a polynomial and the sum in the numerator is taken over some finite set of $\mu\in \mf{t}$. In particular, $u_T$ is a Taylor expansion at zero of a meromorphic function.
\end{thr*}

In Subsection~\ref{SubSecSumOfIntegrals} we get a corollary that works for any twisted trace. Doing a process similar to division with residue and taking Fouier transform we get

\begin{cor*}[\ref{CorTraceIsAnIntegral}]
Any trace $T$ can be written in the form \[T(R)=\sum_{U,D,\zeta}\int_{U+i\zeta}(DR)(ix)w_{U,D,\zeta}(x-i\zeta)dx,\] where $U$ is a real subspace, $D$ is a differential operator and $w_{U,D,\zeta}$ is an exponentially decaying function holomorphic on $U\times B_{\eps}(0)\in U_{\C}$.
\end{cor*}

In Section~\ref{SecPositive} we classify positive traces for any flavors. In Subsection~\ref{SubSecGeneratingFunctionBounded} we prove that for any positive trace $T$ the function $u_T$ is bounded on $\mf{t}_{\R}^*$. After that we prove that $T$ can be written as $\sum_{L\subset \mf{t}_{\R}}\int_{L} R(iz)w_L(z)dz$, where $w_L$ may have poles, so the integral should taken over a small imaginary shift of $L$.

In Subsection~\ref{SubSecAnswer} we finish classification of positive traces. We prove that $w_L$ has no poles, then we show that each functional $T_L(R)=\int_L R(\rmi z)w_L(z)dz$ is a twisted trace on $\mc{A}$ and then we show that each $T_L$ is positive. Combining this with Proposition~\ref{PropPositivityOnTermsOfWeight} completes the classification:

\begin{thr*}[\ref{ThrAnswer}]
Let $\mc{A}$ be an abelian Coulomb branch corresponding to weights $\xi_1,\ldots,\xi_n$ and flavor parameters $b_1,\ldots,b_n$ and $\rho$ be the antilinear automorphism of $\mc{A}$ corresponding to $\zeta\in\mf{t}_{\R}^*$. 

Let $F$ be the set of subspaces $U$ of $\mf{t}_{\R}$ such that there exists a homomorphism of algebras $\mc{A}\to\mc{A}_U$ that sends  $r^{\lambda}R(z)$ to zero if $\lambda\notin U$ and to $r^{\lambda}R(z)|_U$ else. Here $\mc{A}_U$ is the abelian Coulomb branch corresponding to weights $\xi_1^U=\xi_1+U^{\perp},\ldots,\xi_n^U=\xi_n+U^{\perp}$ and the same flavors. For each $\mc{A}_U$ we take the symmetrized fundamental polytope for ``good'' weights \[\mc{P}_U=\sum_{\abs{\Re b_i}<\tfrac12} [-\tfrac12,\tfrac12]\xi_i^U,\] intersect $\mc{P}_U^{\circ}+\zeta$ with $(2\Z)^{\dim U}\subset U^*$ and take convex hull. Denote the result by $\mc{P}_U^+$. 

If $\mc{P}_{\mf{t}_{\R}}^+$ is empty, there are no positive traces. Else, the cone of positive traces has dimension equal to the sum over $U\in F$ of number of integer points in $\mc{P}_U^+$.
\end{thr*}
\section{Setup, notation, analytic formula}
\label{SecSetup}
\subsection{Quantized abelian Coulomb branches}
\label{SubSecAbelianCoulombBranches}

Let $N$ be a representation of $T=(\C^*)^d$. Suppose that $\dim N=n$ and the characters are $\xi_1,\ldots,\xi_n$. Denote the Lie algebra of $T$ by $\mf{t}$, let $Y\subset \mf{t}$ be its cocharacter lattice. Let $b_1,\ldots,b_n\in \C$ be flavor parameters.

The definition of quantized abelian Coulomb branch is given in~\cite{BFN},~Section 4. This is an algebra $\mc{A}$ that is a free $\CN[\mf{t}]$-module, with the basis elements $r^{\lambda}, \lambda\in Y$. The relations are as follows. The action of $\CN[\mf{t}]=S(\mf{t}^*)$ is grading. Namely, for $\lambda\in Y\subset \mf{t}$, $\alpha\in \mf{t}^*$, we have
\[[r^{\lambda},\alpha]=\alpha(\lambda)r^{\lambda}.\] This is equation~(4.8) in~\cite{BFN}. Therefore the adjoint action of $\mf{t}^*$ on $\mc{A}$ has eigenvalues in $Y$ and the $\lambda$-eigenspace is $r^{-\lambda}\C[\mf{t}]$. We will also call $\mf{t}^*$-eigenspaces weight spaces. 

The multiplication of basis elements is
\[r^{\lambda}r^{\mu}=\prod_{i=1}^n \tilde{A_i}(\lambda,\mu)r^{\lambda+\mu},\] where $\tilde{A_i}$ is given by formula~(4.9) in~\cite{BFN} with $\hbar=1$. We will mostly use this formula for $\lambda=-\mu$, in this case the formula becomes
\[r^{\lambda}r^{-\lambda}=\prod_{i=1}^n\prod_{j=1}^{\abs{\xi_i(\lambda)}}\big(\xi_i+b_i+\sgn(\xi_i(\lambda))(\abs{\xi_i(\lambda)}-j+\tfrac 12)\big).\] In other words, in $i$-th term we add all half-integer numbers in $[0,\xi_i(\lambda)]$ or $[\xi_i(\lambda),0]$ to $\xi_i+b_i$ and multiply these elements together. In case $\xi_i(\lambda)=0$ the $i$-th term is one.

A short computation shows that this formula can be rewritten in a more symmetric way:
\[r^{\lambda}r^{-\lambda}=S_{\tfrac12\lambda}\bigg(\prod_{i=1}^n\big(\prod_{j=1}^{\abs{\xi_i(\lambda)}}(\xi_i+b_i+\abs{\xi_i(\lambda)}-j+\tfrac 12)\bigg),\]
where $S_{\lambda}$ is an automorphism of $S(\mf{t}^*)$ defined on $\mf{t}^*$ by
\[S_{\lambda}(\alpha)=\alpha+\alpha(\lambda),\] or, equivalently,
\[(S_{\lambda}f)(\mu)=f(\lambda+\mu).\]

We will consider only automorphisms $g$ of $\mc{A}$ that are identity on $\C[\mf{t}]$. Such $g$ act as a number on each weight space, hence they correspond to group homomorphisms $Y\to \C^*$. Taking a logarithm of $g$ on a basis of $Y$ we write $g=e^{2\pi \rmi \zeta}$, where $\zeta\in \Hom_{\ZN}(Y,\C)=\mf{t}^*$ is FI parameter in physical terms.

We will consider only antilinear automorphisms $\rho$ that sends each weight space to a weight space with the opposite weight and act as minus identity on $\mf{t}^*\subset \C[\mf{t}]$.

Suppose that for each unique $\xi=\xi_{i_1},\ldots,\xi_{i_k}$ 
the corresponding flavors parameters $b_{i_1},\ldots,b_{i_k}$ can be decomposed into pairs $i_s,i_t$ and singletons $i_s=i_t$ with $b_{i_s}=-\ovl{b_{i_t}}$. 

Let $c_{\lambda},\lambda\in Y$ be complex numbers with unit length such that 
\begin{equation}
\label{EqConditionOnC}
c_{\lambda}c_{\mu}=(-1)^{\sum_{i=1}^n d(\xi_i(\lambda),\xi_i(\mu))}c_{\lambda+\mu},
\end{equation} where $d$ is the function from formula~(4.9) in~\cite{BFN}. Then we can define $\rho$ by $\rho(r^{\lambda})=c_{\lambda} r^{-\lambda}$. Let us check that $\rho$ is indeed an automorphism in this case. The commutation relation $[r^{\lambda},\alpha]=\alpha(\lambda)r^{\lambda}$ with $\C[\mf{t}]$ is preserved, so it remains to check that \[\rho(r^{\lambda}r^{\mu})=\rho(r^{\lambda})\rho(r^{\mu}).\] The numbers $c$ cancel each other out and there remains \[\ovl{\tilde{A}(\lambda,\mu)}|_{y\mapsto -y}r^{-\lambda-\mu}\] on the left-hand side and $(-1)^{\sum_{i=1}^n d(\xi_i(\lambda),\xi_i(\mu))}\tilde{A}(-\lambda,-\mu)r^{-\lambda-\mu}$ on the right-hand side. Compare $i_s$-th term on the left with $i_t$-th term on the right, denote the corresponding flavor parameters by $b$ and $b'=-\ovl{b}$. Assume that $\xi(\lambda)\geq 0\geq \xi(\mu)$, the other case is similar. Then by formula~(4.9) in~\cite{BFN} we should check that
\begin{align*}
\prod_{j=1}^{d(\xi_i(\lambda),\xi_i(\mu))}\big(\ovl{b}-\xi_i+(\xi_i(\lambda)-j+\tfrac12)\big)=\\
(-1)^{d(\xi_i(\lambda),\xi_i(\mu))}\prod_{j=1}^{d(-\xi_i(\lambda),-\xi_i(\mu)}\big(b'+\xi_i+(\xi_i(-\lambda)+j-\tfrac12)\big).
\end{align*} Indeed, each of the terms in the product on the right is flipped compared to the term of the left. 

For any $c$ satisfying~\eqref{EqConditionOnC} we can choose $\zeta\in\mf{t}^*_{\R}$ such that \[c_{\lambda}=e^{-\pi \rmi\zeta(\lambda)}(-1)^{\sum_{i=1}^n\max(\xi_i(\lambda),0)}.\] Then \begin{align*}
g(r^{\lambda})=\rho^2(r^{\lambda})=(-1)^{\sum_{i=1}^n \xi_i(\lambda)}c_{-\lambda}\ovl{c_{\lambda}}r^{\lambda}=\\
(-1)^{\sum_{i=1}^n \xi_i(\lambda)}c_{-\lambda}^2r^{\lambda}=e^{2\pi\rmi\zeta(\lambda)+\sum_{i=1}^n\pi\rmi\xi_i(\lambda)}r^{\lambda}.
\end{align*} Shifting $\zeta$ by $\tfrac12\sum_{i=1}^n\xi_i$ gives $g(r^{\lambda})=e^{2\pi\rmi\zeta}r^{\lambda}$. To each $g$ correspond $2^d$ choices of $\rho$. 

Combining expression for $c_{\lambda}$ with the definition of $\rho$ we get
\begin{equation}
\label{EqDefinitionOfRho}
\rho(r^{\lambda}R(z))=(-1)^{\sum_{i=1}^n\max(\xi_i(\lambda),0)}e^{-\pi \rmi\zeta(\lambda)}r^{-\lambda}\ovl{R}(-z).
\end{equation}

\subsection{Trace condition}
\label{SubSecTraceCondition}
Let $T$ be a $g$-twisted trace on $\A$ with $g$ corresponding to $\zeta$ as above. We want to describe $T$ more explicitly.

Suppose that $b=r^{\lambda}R(x)$, $R\in \C[\mf{t}]$ is an element of degree $\lambda\neq 0$. Choose $\alpha$ such that $\alpha(\lambda)\neq 0$. Using $T(ab)=T(bg(a))$ for $a=\alpha$ and $b$ we get $T([a,b])=0$, hence $T(b)=0$. It follows that any $g$-twisted trace $T$ is uniquely defined by its values on $\C[\mf{t}]$.  Abusing notation we will denote the restriction of $T$ to $\C[\mf{t]}$ also by $T$.

Let us write a condition for a linear function $T\colon \C[\mf{t}]\to \C$ to give a twisted trace on $\mc{A}$. It is enough to check the trace condition $T(ab)=T(bg(a))$ for $a,b$ in weight spaces. Since $T$ is supported on $\C[\mf{t}]$, both sides are zero unless $a,b$ have opposite weights. Let $a=r^{\lambda}f$, $b=r^{-\lambda}h$, where $f,h\in S(\mf{h}^*)$.
Note that \[r^{\lambda}f=S_{\lambda}(f)r^{\lambda}\] for all $f\in S(\mf{t}^*)$, where $S_{\lambda}$ is the shift of argument automorphism defined above. We also have \[r^{\lambda}r^{-\lambda}=S_{\tfrac12\lambda}\bigg(\prod_{i=1}^n\big(\prod_{j=1}^{\abs{\xi_i(\lambda)}}(\xi_i+b_i+\abs{\xi_i(\lambda)}-j+\tfrac 12)\bigg).\]

Denote the expression in brackets by
\[P_{\lambda}=\prod_{i=1}^n\prod_{j=1}^{\abs{\xi_i(\lambda)}}(\xi_i+b_i+\tfrac12\abs{\xi_i(\lambda)}-j+\tfrac12).\]

Then
\[ab=fr^{\lambda}hr^{-\lambda}=fS_{\lambda}(h)r^{\lambda}r^{-\lambda}=fS_{\lambda}(h)S_{\tfrac12\lambda}(P_{\lambda}),\]
\[bg(a)=e^{2\pi\rmi\zeta(\lambda)} hr^{-\lambda}fr^{\lambda}=e^{2\pi\rmi\zeta(\lambda)} hS_{-\lambda}(f)r^{-\lambda}r^{\lambda}=e^{2\pi\rmi\zeta(\lambda)} hS_{-\lambda}(f)S_{-\tfrac12\lambda}(P_{\lambda}),\]

Note that \[fS_{\lambda}(h)=S_{\lambda}(hS_{-\lambda}(f))\] and this expression can be any element of $\C[\mf{t}]$.
So we proved the following
\begin{prop}
\label{PropTraceConditionAbstract}
$T\colon\C[\mf{t}]\to \C$ defines a trace if and only if for all $\lambda\in Y$, $R\in \C[\mf{t}]$ we have 
\begin{equation}
\label{EqTraceConditionAbstract}
T(S_{\tfrac12\lambda}(RP_{\lambda}))=e^{2\pi\rmi\zeta(\lambda)}T(S_{-\tfrac12\lambda}(RP_{\lambda})).
\end{equation}
\end{prop}
From now on, we will think of traces as maps $T\colon \C[\mf{t}]\to \C$ satisfying~\eqref{EqTraceConditionAbstract}.

It follows from multiplication formula and Proposition~\ref{PropTraceConditionAbstract} that the sets of twisted and positive traces does not change if each $P_{\lambda}$ is multiplied by a positive number. Suppose that $\xi_i=k\xi$ for a positive integer $k$. Then we change $\xi_i, b_i$ to $k$ copies of $\xi$ with flavors $\frac{b_i}{k}$, $\frac{b_i+1}{k}$, $\cdots$, $\frac{b_i+k-1}{k}$. This may change the algebra, but each $P_{\lambda}$ will be multiplied by a positive number, so the sets of twisted and positive traces do not change. Also, changing $\xi_i,b_i$ to $-\xi_i,-b_i$ does not change the multiplication in $\mc{A}$.

So we can assume that each $\xi_i$ is not a multiple of a character and we never have $\xi_i=-\xi_j$. 

Suppose that $\{\xi_1,\ldots,\xi_n\}$ is unimodular. This means the following: if some $\xi_{i_1},\ldots,\xi_{i_d}$ is a $\C$-basis of $\mf{t}_{\R}$, then it is a $\Z$-basis of the character lattice. In particular, this holds when $N$ is a representation space of a quiver and $T=\prod (\C^{\times})^{\dim V_i}\subset G$ is a maximal torus.

In the case of generic $\zeta$ and unimodular set of weights, we can compute the upper bound on the dimension of the space of traces using~\eqref{EqTraceConditionAbstract} as follows. The algebra $\gr\mc{A}$ is a non-quantized abelian Coulomb branch. Consider the linear subspace $I$ of $\gr\mc{A}$ generated by the leading terms of the elements \[S_{\tfrac12\lambda}(RP_{\lambda})-e^{2\pi\rmi\zeta(\lambda)}S_{-\tfrac12\lambda}(RP_{\lambda}).\] The leading term of this element equals to  $\ovl{R}\prod_{i=1}^n\xi^{\abs{\xi(\lambda)}}$, where $\ovl{R}$ is the leading term of $R$. Hence $I$ is an ideal. In general case $I$ is too small: already for $\xi_1=e_1,\xi_2=e_2,\xi_3=ae_1+be_2$ the codimension is too large. 

In unimodular case we need the following lemma:
\begin{lem}
\label{LemCanTakeMaximalFaces}
Ideal $I$ is generated by elements $\prod_{i=1}^n\xi^{\abs{\xi(\lambda)}}$ with $\lambda$ corresponding to a face $F$ of $\mc{P}=\sum_{i=1}^n [0,1]\xi_i$. This means that if $H$ is the hyperplane corresponding to $F$, then $\lambda$ is a generator of $H^{\perp}\cap Y$.
\end{lem}
\begin{proof}
Let $F$ be a face of $\mc{P}=\sum [0,1]\xi_i$ of dimension $d-1$. Then $F\cong \sum_{i\in S} [0,1]\xi_i$ for some subset $S$. Take $d-1$ linearly independent vectors $\xi_{i_1},\ldots,\xi_{i_{d-1}}$ with $i_1,\ldots,i_{d-1}\in S$. Denote the hyperplane they span by $H$. Take some $\xi_j$ that does not belong to $H$. Using unimodularity we get that $\xi_{i_1},\ldots,\xi_{i_{d-1}},\xi_j$ form a $\Z$-basis of $\mf{t}_{Z}^*$. Let $\lambda$ be the last element of the dual basis: $\lambda(\xi_j)=1$, $\lambda(H)=\{0\}$. Then $\lambda$ belongs to $\mf{t}_{\Z}=Y$. Since $j$ can be any element of $\{1,\ldots,n\}\setminus S$ we get that any generator $\lambda$ of $H_{\Z}^{\perp}\subset Y$ satisfies $\lambda(\xi_j)=\pm 1$. 

Because of this, for any $\mu\in Y$ that attains its maximum on $\mc{P}$ on a face $F'$ of dimension less than $d-1$, the element $\prod_{i=1}^n \xi_i^{\abs{\xi_i(\mu)}}$ is divisible by $\prod_{i=1}^n \xi_i^{\abs{\xi_i(\lambda)}}$, where $\lambda$ corresponds to a face $F$ of dimension $d-1$ that contains $F'$. 
\end{proof} 

Now we are ready to prove
\begin{prop}
Order $\xi_1,\ldots,\xi_n$ so that $\xi_1,\ldots,\xi_d$ span $\mf{t}^*$. Define the set of subsets $B(\xi_1,\ldots,\xi_n)\subset \mc{P}(\{1,2,\ldots,n\})$ inductively by 
\begin{enumerate}
\item
When $d=0$ we set $B=\{\emptyset\}$.
\item
$B(\xi_1,\ldots,\xi_d)=\{\emptyset\}$.
\item
$B(\xi_1,\ldots,\xi_{n+1})=B(\ovl{\xi_1},\ldots,\ovl{\xi_n})\cup \bigcup_{S\in B(\xi_1,\ldots,\xi_n)}S\sqcup\{n+1\}$. Here $\ovl{\xi_i}$ means the image of $\xi_i$ in $S(\mf{t}^*/\xi_{n+1})=\C[\xi_{n+1}^{\perp}]$, and we throw away zeroes.
\end{enumerate} Then the images of elements $m_S=\prod_{i\in S}\xi_i, S\in B$ span $\C[\mf{t}]/I$. Moreover, the size of $B$ equals to the number of linearly independent subsets of size $d$ in $\{\xi_1,\ldots,\xi_n\}$.
\end{prop}
\begin{proof}
We prove this statement using induction on $d$ and then on $n\geq d$. In the case $d=0$ we have $\C[Y]=\C$, $I=0$, $m_{\emptyset}=1$, and $\emptyset$ is a unique linearly independent subset of size $0$.

For $n=d$ we have $I=(\xi_1,\ldots,\xi_d)$, it has codimension one with complement spanned by $m_{\emptyset}=1$ and there is exactly one linearly independent subset of size $d$ in $\xi_1,\ldots,\xi_d$.

Now we do induction step. Let $R\in \C[\mf{t}]$ be a polynomial with image $\ovl{R}$ in $\C[\xi_{n+1}^{\perp}]$. Let $B_1$ be the set corresponding to $\ovl{\xi_1},\ldots,\ovl{\xi_n}$. We can write $\ovl{R}=\sum_{S\in B_1}c_S m_S+\sum_{\lambda}c_{\lambda}\prod_{i=1}^n \ovl{\xi_i}^{\abs{\xi_i(\lambda)}}$. Here we take elements $\lambda\in\mf{t}$ with $\lambda(\xi_{n+1})=0$. Therefore \[R=\sum_{S\in B_1}c_Sm_S+\sum_{\lambda}c_{\lambda}\prod_{i=1}^n \ovl{\xi_i}^{\xi_i(\lambda)}+\xi_{n+1}Q.\] Let $B_2$ be the set corresponding to $\xi_1,\ldots,\xi_n$. Then we can write \[Q=\sum_{S\in B_2}c_Sm_S+\sum_{\lambda}c_{\lambda}\prod_{i=1}^n \ovl{\xi_i}^{\abs{\xi_i(\lambda)}}.\] Note that all $\lambda$ in this sum correspond to some face of $\sum[0,1]\xi_i$. Then either $\xi_{n+1}(\lambda)=0$ or we can use the proof of Lemma~\ref{LemCanTakeMaximalFaces} to get $\xi_{n+1}(\lambda)=1$. In both cases $\abs{\xi_{n+1}(\lambda)}\leq 1$. Hence $\xi_{n+1}\prod_{i=1}^n \ovl{\xi_i}^{\abs{\xi_i(\lambda)}}\in I$. It follows that the images of $m_S, S\in B_1$ and $\xi_{n+1}m_S, S\in B_2$ span $\C[\mf{t}]/I$.

A linearly independent subset of size $d$ in $\{\xi_1,\ldots,\xi_{n+1}\}$ either contains $\xi_{n+1}$ or not, in the first case there are $\abs{B_1}$ such subsets, in the second case there are $\abs{B_2}$ such subsets.
\end{proof}
We showed only upper bound for the dimension of $\C[\mf{t}]/I$. This also gives an upper bound for the dimension of the space of traces. We will show that the number of traces is at least the number of linearly independent subsets of size $d$. 

\begin{rem}
One can think of $B$ as follows. First, we decompose the fundamental polytope $P_n=\sum_{i=1}^d[0,1]\xi_i$ into parallelotopes by induction using $P_{n+1}=(P_n+\xi_{n+1})\cup ([0,1]\xi_{n+1}+Q_n)$, where $Q_n$ is a subset of $\R^d$ formed by certain faces of $P_n$. Then we draw a rooted tree with vertices corresponding to the parallelotopes and labeled edges drawn inductively. One part of the tree with the root comes from $Q_n$. Another comes from $P_n$ and we also draw an edge in the direction of $\xi_{n+1}$ with label $n+1$. Then for each vertex we consider the path to the root. All of the labels on the edges are different and we get a subset of $\{1,\cdots,n+1\}$. The set $B$ consists of all subsets that can be obtained in this way, hence the size of $B$ equals to the volume of $P_{n+1}$. In~\cite{AP} Arbo and Proudfoot give an abstract definition of a hypertoric variety and construct a hypertoric variety from any parallelotope tiling of $P_n$. We did not investigate whether all such tilings can be turned into trees and bases of $\C[\mf{t}]/I$ or the tiling above is special in this sense.
\end{rem}
\subsection{Analytic formulas for traces}
\label{SubSecAnalyticFormulas}
Our main goal is to classify positive traces, in order to do this we want to write traces as integrals plus some other terms. Similarly to~\cite{EKRS} we start with the nicest case. Suppose that all $b_j$ satisfy $\abs{\Re b_j}< \tfrac12$. Let $w(x)$ be a meromorphic function on $\mf{t}$ such that
\begin{enumerate}
\item
$w(x)$ exponentially decays at infinity.
\item
We have $w(x+\rmi\lambda)=e^{2\pi\rmi\zeta(\lambda)}w(x)$ for all $\lambda\in Y$. 
\item
$\prod_{j=1}^n (e^{2\pi\xi_j}+e^{2\pi\rmi b_j})w(x)$ is entire function.
\end{enumerate} 
This means that $w(x)=e^{2\pi\zeta}\frac{\sum c_a e^a}{\prod_{j=1}^n (e^{2\pi \xi_j}+e^{2\pi \rmi b_j})}$, where the sum in the numerator is taken over $a\in 2\pi Y^{\wedge}$ that belong to the interior of $-\zeta+\sum_{i=1}^n \xi_i[0,1]$.

\begin{defn}
We say that the Minkowski sum of intervals (a hypercube projection) $\sum_{i=1}^n \xi_i[0,1]$ is the fungamental polytope of $\mc{A}$.
\end{defn}

\begin{prop}
\label{PropTraceAsAnIntegral}
In this case $T(R)=\int_{\R^d}R(\rmi x)w(x)dx$ defines a trace on $\C[\mf{t}]$.
\end{prop}

\begin{proof}
Exponential decay of $w$ shows that $T$ is a well-defined linear function.

We need to check that \[T\big(R(x+\frac{\lambda}{2})P_{\lambda}(x+\frac{\lambda}{2})-e^{2\pi\rmi\zeta(\lambda)}R(x-\frac{\lambda}{2})P_{\lambda}(x-\frac{\lambda}{2})\big)=0.\] Using periodicity, exponential decay of $w$ and assuming that $R(\rmi x)P_{\lambda}(\rmi x)w(x-\rmi\frac{\lambda}{2})$ is holomorphic on $\R^n+i[-\frac{\lambda}{2},\frac{\lambda}{2}]$ we get

\begin{align*}
\int_{\R^d}R(\rmi x+\frac{\lambda}{2})P_{\lambda}(\rmi x+\frac{\lambda}{2})w(x)dx&=\int_{\R^d-\rmi\frac{\lambda}{2}}R(\rmi x)P_{\lambda}(\rmi x)w(x+\rmi\frac{\lambda}{2})dx=\\
e^{2\pi\rmi\zeta(\lambda)}\int_{\R^d-\rmi\frac{\lambda}{2}}R(\rmi x)P_{\lambda}(\rmi x)w(x-\rmi\frac{\lambda}{2})dx&=e^{2\pi\rmi\zeta(\lambda)}\int_{\R^d+\rmi\frac{\lambda}{2}}R(\rmi x)P_{\lambda}(\rmi x)w(x-\rmi\frac{\lambda}{2})dx\\
&=e^{2\pi\rmi\zeta(\lambda)}\int_{\R^d}R(\rmi x-\frac{\lambda}{2})P_{\lambda}(\rmi x-\frac{\lambda}{2})w(x)dx,
\end{align*}
hence
\[T\big(R(x+\frac{\lambda}{2})P_{\lambda}(x+\frac{\lambda}{2})\big)=e^{2\pi\rmi\zeta(\lambda)}T\big(R(x-\frac{\lambda}{2})P_{\lambda}(x-\frac{\lambda}{2})\big).\]

It remains to check that the function $R(\rmi x)P_{\lambda}(\rmi x)w(x-\rmi\frac{\lambda}{2})$ is indeed holomorphic on a neighborhood of $\R^n+i[-\frac{\lambda}{2},\frac{\lambda}{2}]$. 

Since $\prod_{j=1}^n (e^{2\pi\xi_j}+e^{2\pi\rmi b_j})w(x)$ is entire function, \[\prod_{j=1}^n (e^{2\pi \xi_j-\pi\rmi\xi_j(\lambda)}-e^{2\pi\rmi b_j})w(x-\rmi\frac{\lambda}{2})\] is an entire function. The function $e^{2\pi \xi_j-\pi\rmi\xi_j(\lambda)}+e^{2\pi\rmi b_j}$ has simple zeroes along the union of hyperplanes $\xi_j(x)\in \rmi(\frac{1}{2}+b_j+\frac{\xi_j(\lambda)}{2}+\ZN)$.

Let us compute the intersection of this set with $\R^d+\rmi\lambda[-\frac{1}{2},\frac{1}{2}]$. If $x\in \R^n+i[-\frac{\lambda}{2},\frac{\lambda}{2}]$ then $\Im \xi_j(x)\in [-\tfrac12\abs{\xi_j(\lambda)},\tfrac12\abs{\xi_j(\lambda)}]$. Using that $\Re b_j\in (-\tfrac12,\tfrac12)$ we get $\xi_j(x)\in \{\rmi(\frac{1}{2}+b_j-\frac{\abs{\xi_j(\lambda)}}{2}),\rmi(\frac{3}{2}+b_j-\frac{\abs{\xi_j(\lambda)}}{2}),\cdots, \rmi(\frac{\abs{\xi_j(\lambda)}}{2}-\frac{1}{2}+b_j)\}$.

Hence the function \[\frac{e^{2\pi \xi_j}+e^{2\pi\rmi b_j+\pi\rmi \xi_j(\lambda)}}{\prodl_{k=1}^{\abs{\xi_j(\lambda)}}\big(\xi_j-\rmi(b_j+\tfrac12+\frac{\abs{\xi_j(\lambda)}}{2}-k)\big)}\] vanishes nowhere on $\R^d\times [\frac{-\rmi\lambda}{2},\frac{\rmi\lambda}{2}]$.

Taking product for all $j$ we get that
\[\frac{\prod_{j=1}^n (e^{2\pi \xi_j+\pi\rmi\xi_j(\lambda)}-e^{2\pi\rmi b_j})}{P_{\lambda}(\rmi x)}\] vanishes nowhere on $\R^d\times [\frac{-\rmi\lambda}{2},\frac{\rmi\lambda}{2}]$. It follows that $P_{\lambda}(\rmi x)w(x+\frac{\lambda}{2})$ is holomorphic on a neighborhood $\R^d\times [\frac{-\rmi\lambda}{2},\frac{\rmi\lambda}{2}]$.
\end{proof}
We can also say when such trace is positive. Recall that we can change the weight $\xi_i$ to its opposite without changing the algebra. Hence we can assume that there exists $\tau\in \mf{t}$ such that $\xi_i(\tau)<0$ for $i=1,\ldots,n$. We also need to shift $\zeta$ by $\tfrac12\sum_{i=1}^n\xi_i$.
\begin{prop}
\label{PropPositivityOnTermsOfWeight}
Let $T$ be the trace constructed in Proposition~\ref{PropTraceAsAnIntegral} with weight $w(x)=e^{2\pi\zeta-\pi\sum_{i=1}^n\xi_i}\frac{Q(e^{2\pi z_1},\ldots, e^{2\pi z_d})}{\prod_{j=1}^n (e^{2\pi \xi_j}+e^{2\pi \rmi b_j})}$. Then $T$ is positive if and only if $Q\neq 0$ is a  polynomial that has nonnegative values on $\R^d$.
\end{prop}
\begin{proof}

When $\lambda\neq \mu$, the subspaces $r^{\lambda}\C[\mf{t}]$ and $r^{\mu}\C[\mf{t}]$ are orthogonal to each other with respect to the form $(a,b)=T(a\rho(b))$. Hence it is enough to check the positivity condition $T(a\rho(a))>0$ for elements $a$ of the form $r^{\lambda}R(z)$. In this case using~\eqref{EqDefinitionOfRho} we get
\begin{align*}
a\rho(a)=(-1)^{\sum_{i=1}^n\max(\xi_i(\lambda),0)}r^{\lambda}R(z) e^{-\pi \rmi\zeta(\lambda)}r^{-\lambda}\ovl{R}(-z)=\\
(-1)^{\sum_{i=1}^n\max(\xi_i(\lambda),0)}e^{-\pi\rmi\zeta(\lambda)}P_{\lambda}(z+\frac{\lambda}{2})R(z+\lambda)\ovl{R}(-z).
\end{align*}
Hence
\begin{align*}
T(a\rho(a))=(-1)^{\sum_{i=1}^n\max(\xi_i(\lambda),0)}e^{-\pi\rmi\zeta(\lambda)}\int_{\R^d}P_{\lambda}(\rmi z+\frac{\lambda}{2})R(\rmi z+\lambda)\ovl{R}(-\rmi z)w(z)dz.
\end{align*} Since $P_{\lambda}(\rmi z)w(z+\frac{\lambda}{2})$ is holomorphic on a neighborhood $\R^d\times [\frac{-\rmi\lambda}{2},\frac{\rmi\lambda}{2}]$ we can shift the contour to $\R^d+\rmi\frac{\lambda}{2}$ and make a change of variable $z\mapsto z+\rmi\frac{\lambda}{2}$, giving
\[T(a\rho(a))=(-1)^{\sum_{i=1}^n\max(\xi_i(\lambda),0)}e^{-\pi\rmi\zeta(\lambda)}\int_{\R^d}P_{\lambda}(\rmi z)R(\rmi z+\frac{\lambda}{2})\ovl{R}(-\rmi z+\frac{\lambda}{2})w(z+\rmi\frac{\lambda}{2})dz.\] Using Lemma~\ref{LemApproximation} this is equivalent \begin{equation}
\label{EqWShiftedIsPositive}
(-1)^{\sum_{i=1}^n\max(\xi_i(\lambda),0)}e^{-\pi\rmi\zeta(\lambda)}P_{\lambda}(\rmi z)w(z+\rmi \frac{\lambda}{2})dz\geq 0
\end{equation} for all $z\in \R^d$. 

By our assumption on flavors, $(\rmi)^{\sum_{i=1}^n \xi_i(\lambda)}P_{\lambda}(\rmi z)$ is real on $\R^d$. A term $\rmi\xi_i+b_i+\frac{\abs{\xi_i(\lambda)}}{2}-j+\tfrac12$ vanishes on a hyperplane in $\R^d$ if and only if $\Re(b_i+\frac{\abs{\xi_i(\lambda)}}{2}-j+\tfrac12)=0$. Also, the denominator $\prod_{i=1}^n (e^{2\pi\xi_i+\pi\rmi\xi_i(\lambda)}+e^{2\pi\rmi b_i})$ of $w(z+\rmi\frac{\lambda}{2})$ is real on $\R^d$. A term $e^{2\pi\xi_i+\pi\rmi\xi_i(\lambda)}+e^{2\pi\rmi b_i}$ vanishes on a hyperplane in $R^d$ if and only if $\Im(\pi\rmi\xi_i(\lambda))\in\Im(2\pi\rmi b_i)+\pi+2\pi\Z$, this means $\frac{\xi_i(\lambda)}{2}\in \Re b_i+\tfrac12+\Z$. In this case, $\abs{\Re b_i}<\tfrac12$ we can always find $1\leq j\leq \abs{\xi_i(\lambda)}$ such that $\Re(b_i)+\frac{\abs{\xi_i(\lambda)}}{2}-j+\tfrac12=0$. It follows that 
\begin{equation}
\label{EqPLambdaOverDenominator}
\frac{P_{\lambda}(\rmi z)}{\prod_{i=1}^n (e^{2\pi\xi_i-\pi\rmi\xi_i(\lambda)}+e^{2\pi\rmi b_i})}
\end{equation} always has the same argument on $\R^d$. 

Take $\tau$ as above, $\xi_i(\tau)<0$ for $i=1,\ldots,n$. Taking the limit $z=\tau r$, $r\to\infty$, the denominator of~\eqref{EqPLambdaOverDenominator} tends to $e^{2\pi\rmi(\sum_{i=1}^n b_i)}$. This expression is greater than zero, since $\sum b_i$ is imaginary. The numerator has leading term \[\prod_{i=1}^n (\xi_i(\rmi \tau r))^{\abs{\xi_i(\lambda)}},\] it has argument $(-\rmi)^{\sum_{i=1}^n \abs{xi_i(\lambda)}}$. 

We have \[w(z+\rmi\frac{\lambda}{2})=e^{2\pi\zeta+\pi\rmi\zeta(\lambda)-\pi\sum_{i=1}^n\xi_i-\frac{\pi}{2}\rmi\sum_{i=1}^n \xi_i(\lambda)}\frac{Q(e^{2\pi z_1}e^{\pi\rmi\lambda_1},\ldots,e^{2\pi z_d}e^{\pi\rmi\lambda_d})}{\prod_{j=1}^n (e^{2\pi \xi_j+\pi\rmi\xi_j(\lambda)}+e^{2\pi \rmi b_j})}.\] The poles of denominator cancel out the roots of $P_{\lambda}(\rmi z)$. The term $e^{\pi\rmi\zeta(\lambda)}$ cancels out the first term in~\eqref{EqWShiftedIsPositive} and $e^{2\pi\zeta}$ is positive on $\R^d$. Taking into account the signs and arguments above, the equation~\eqref{EqWShiftedIsPositive} is satisfied if and only if \[(-\rmi)^{\sum_{i=1}^n \abs{\xi_i(\lambda)}}(-1)^{\sum_{i=1}^n\max(\xi_i(\lambda),0)}e^{-\frac{\pi}{2}\rmi\sum_{i=1}^n \xi_i(\lambda)}Q(e^{2\pi z_1}e^{\pi\rmi\lambda_1},\ldots,e^{2\pi z_d}e^{\pi\rmi\lambda_d})\geq 0\] on $\R^d$. Note that \[(-\rmi)^{\sum_{i=1}^n \abs{\xi_i(\lambda)}}(-1)^{\sum_{i=1}^n\max(\xi_i(\lambda),0)}e^{-\frac{\pi}{2}\rmi\sum_{i=1}^n \xi_i(\lambda)}=1.\]  The range of $e^{2\pi z_i+\pi\rmi\lambda_i}$ is $(-1)^{\lambda_i}R_{>0}$. Hence the range of \[Q(e^{2\pi z_1}e^{\pi\rmi\lambda_1},\ldots,e^{2\pi z_d}e^{\pi\rmi\lambda_d})\] equals to the range of $Q(y)$ when $y_i\in (-1)^{\lambda_i}R_{>0}$. 

This gives all possible combinations of signs of $y_1\neq 0$, $\ldots$, $y_d\neq 0$. We get that $T$ is positive if and only if $Q(y)$ is nonnegative for $y\in \R^d$.

What is the dimension of the cone of such polynomials? On one hand, all extremal terms of $Q$ should have all monomial degrees even. On the other hand, let $S\subset (2\Z)^d$ be a subset and $\mc{P}_+$ be the convex hull of $S$. It is not hard to see that the nonnegative polynomials that have all terms in $\mc{P}_+$ form a convex cone of dimension equal to the number of integer points in $\mc{P}_+$. In our case we should take $S=\mc{P}^{\circ}\cap (2\Z)^d$. This gives the answer: possible weights $w$ form a cone of dimension equal to the number of integer points in the convex hull of $\mc{P}^{\circ}\cap (2\Z)^d$.

\end{proof}
\begin{rem}
It can be checked that in the case $d=1$ the number of integer points in $\mc{P}^{\circ}\cap (2\Z)^d$ recovers the answer in Theorem 4.6~\cite{EKRS}.
\end{rem}

\section{Exponential generating function of a trace}
\label{SecExpGenerating}
\subsection{Definition and basic properties}
\label{SubSecExpDefinition}

In the case when the flavor parameters do not satisfy $\abs{\Re b_j}<\tfrac12$, the integral in Proposition~\ref{PropTraceAsAnIntegral} needs to be modified. In~\cite{EKRS} we described additional terms explicitly. Here the combinatorics is more complicated, so we use a different strategy.

Let $w(x)$ be the integration weight from Proposition~\ref{PropTraceAsAnIntegral}. Let \[u(y)=\mc{F}w=\int_{\R^d}w(x_1,\ldots,x_d)e^{\rmi\sum x_iy_i}dx_1\cdots dx_d\] be its non-unitary Fourier transform in angular frequency.
 Since $w$ is exponentially decaying at infinity, the function $u(y)$ is holomorphic on $\R^d\times U_{\eps}$ for some small ball $U_{\eps}\subset i\R^d$ around zero. We also have \[(\prod_{j=1}^d\partial_j^{k_j})u(0)=\mc{F}(\prod_{j=1}^d (\rmi x_j)^{k_j} w)|_{y=0}=\int_{\R^d}w(x) \prod_{j=1}^d (\rmi x_j)^{k_j}dx=T(x_1^{k_1}\cdots x_d^{k_d}).\] 

Hence the Taylor series of $u$ at zero can be computed as
\[\sum_{k_1,\ldots,k_d\geq 0}\frac{T(x_1^{k_1}\cdots x_d^{k_d})}{k_1!\cdots k_d!}y_1^{k_1}\cdots y_d^{k_d}.\]

If we consider this expression as a formal power series, it makes sense for any $T$:
\begin{defn}
Let $T$ be a linear map from $\C[\mf{t}]$ to $\C$. Choose a basis $x_1,\ldots,x_d$ of $\mf{t}^*$, let $y_1,\ldots, y_d$ be a dual basis. We define an exponential generating function of $T$ to be the formal power series given by 
\[u(y)=u_T(y)=\sum_{j_1,\ldots,j_m\geq 0}\frac{y_1^{j_1}\cdots y_m^{j_m}}{j_1!\cdots j_m!}T(x_1^{j_1}\cdots x_m^{j_m}).\] 
\end{defn}
 Note that $u$ can be considered an element of $\widehat{\C[\mf{t}^*]}_0$ independent of the choice of basis and that $T$ is uniquely defined by $u$.

For $f,g\in \CN[\mf{t}]$ we define $Tf(g)=T(fg)$. Then $u_{Tx_j}=\partial_j u_T$. 

Let $u$ be the exponential generating function of $T\circ S_{\lambda}$. Then
\[u=e^{\lambda}u_T.\]

Let us rewrite the trace condition~\eqref{EqTraceConditionAbstract} in terms of $u_T$. Since $R$ is any polynomial, this equation is equivalent to
\[(T\circ S_{\tfrac12\lambda})P_{\lambda}=e^{2\pi\rmi\zeta(\lambda)}(T\circ S_{-\tfrac12\lambda})P_{\lambda}.\] Passing to exponential generating functions and applying the rule for differentiation and shift of argument we get
\begin{equation}
\label{EqDiffEqForUT}
D_{\lambda}(e^{\tfrac12\lambda}u_T)=e^{2\pi\rmi\zeta(\lambda)}D_{\lambda}(e^{-\tfrac12\lambda}u_T).
\end{equation}
 It follows that
\[(e^{\tfrac12\lambda}-e^{2\pi\rmi\zeta(\lambda)}e^{-\tfrac12\lambda})u_T\in\Ker D_{\lambda}.\] Note that $P_{\lambda}$ is a product of liner factors, so that $D_{\lambda}$ is a composition of differential operators of first order.

\paragraph{Example: $T=\CN^{\times}$.} In this case the quantized Coulomb branch is generalized Weyl algebra we considered in~\cite{EKRS}. As shown there, we need to consider only $\lambda=\lambda_0$, the fundamental coweight, in equation~\eqref{EqTraceConditionAbstract}. Let $P=P_{\lambda}$ and $D=P(\partial)$. The parameter $\zeta$ is a number in this case, let $s=e^{2\pi \rmi \zeta}$. Then the differential equation~\eqref{EqDiffEqForUT} becomes \[D((e^{\frac{t}{2}}-se^{-\frac{t}{2}})u_T)=0,\] this is equivalent to \[u_T=\frac{\sum Q_ie^{\alpha_it}}{e^{\frac{t}{2}}-se^{-\frac{t}{2}}}.\] Here $\alpha_i$ are distinct roots of $P$ and $Q_i$ is a polynomial of degree at most multiplicity of $\alpha_i$ minus one. Since $u_T$ should be regular at zero, in the case $s=1$ we should have $\sum Q_i(0)=0$. Hence the space of possible functions $u_T$ has dimension $\deg P$ or $\deg P-1$, as in~\cite{EKRS}. We will show below that $u_T$ is bounded for any positive trace $T$. In this case, the function $u_T$ is bounded if and only if $Q_i=0$ when $\abs{\Re\alpha_i}>\tfrac12$ and $Q_i$ is constant when $\abs{\Re\alpha_i}=\tfrac12$. Note that both here and in~\cite{EKRS} the properties of the trace depend on the relation between $\abs{\Re\alpha_i}$ and $\tfrac12$.
\subsection{Generating function is meromorphic}
\label{SubSecMeromorphic}
We will not describe $u_T$ completely in the general case. Let $I$ be the set of all lines that could be obtained as intersection of hyperplanes $\xi_j=0$. For each $l\in I$ choose a shortest coweight $\lambda_l$ in $l$.
\begin{thr}
We have
\label{ThrFourierTransformIsMeromorphic}
\begin{equation}
\label{EqExponentsOverExponents}
u_T=\frac{\sum_{\mu} S_{\mu}e^{\mu}}{\prod_{l\in I}(e^{\tfrac12\lambda_l}-e^{2\pi\rmi \zeta(\lambda_l)}e^{-\tfrac12\lambda_l})},
\end{equation} where $S_{\mu}$ is a polynomial and the sum in the numerator is taken over some finite set of $\mu\in \mf{t}$. In particular, $u_T$ is a Taylor expansion at zero of a meromorphic function.
\end{thr}
\begin{proof}
Let $F=\prod_{l\in I}(e^{\tfrac12\lambda_l}-e^{2\pi\rmi \zeta(\lambda_l)}e^{-\tfrac12\lambda_l})u_T$.

Fix some $l=l_0\in I$. Denote $\lambda_0=\lambda_l$. Let $D_0=D_{\lambda_0}$. We have \[D_0((e^{\tfrac12\lambda_0})-e^{2\pi\rmi \zeta(\lambda_0)}e^{-\tfrac12\lambda_0})u_T)=0.\] Here $D_0=\prod_{j\colon\xi_j(\lambda_0)\neq 0}P_{j}(D_{\xi_j})$ for some polynomials $P_{j}$ in one variable.

We want to get a differential equation for $\prod_{l\in I}(e^{\tfrac12\lambda_l}-e^{2\pi\rmi \zeta(\lambda_l)}e^{-\tfrac12\lambda_l})u_T$. Let $f=(e^{\tfrac12\lambda_0}-e^{2\pi\rmi\zeta(\lambda_0)}e^{-\tfrac12\lambda_0})u_T$. Take any coweight $\mu$. We see that 
\[\bigg(\prod_{\xi_j(\lambda_0)\neq 0}P_j(D_{\xi_j}+\tfrac12\xi_j(\mu))P_j(D_{\xi_j}-\tfrac12\xi_j(\mu))\bigg)(e^{\tfrac12\mu}-e^{2\pi\rmi\zeta(\mu)}e^{-\tfrac12\mu})f=0.\]

Using this for each $\mu=\lambda_{m}, m\in I\setminus\{l\}$ we get
\[D_{I,l}(F)=0,\] where $D_{I,l}$ is the differential operator corresponding to the polynomial
\begin{equation}
\label{EqQIlIsProduct}
Q_{I,l}=\prod_{j\colon\xi_j(\lambda_l)\neq 0}Q_{I,l,j}(\xi_j),
\end{equation} where $Q_{I,l,j}$ are some polynomials in one variable.

Consider ideal $J$ generated by all $Q_{I,l}$. Consider the zero set $Z(J)$ of $J$. Since each $Q_{I,l}$ defines a union of affine hyperplanes $\xi_i=c_j$, $Z(J)$ is a union of affine subspaces and the corresponding linear spaces are intersection of some number of hyperplanes $\xi_j=0$. We claim that each of these subspaces is a point. Assume the opposite, let $U+c\subset Z(J)$, where $U$ is a subspace. By assumption, $\xi_1,\ldots,\xi_n$ generate $\mf{t}$, so intersecting $U$ with some hyperplanes of the form $\xi_j=0$ we can assume that $U=m$ is one-dimensional. By definition $m\in I$. From~\eqref{EqQIlIsProduct} we get that $Z(Q_{I,m})$ is a union of hyperplanes $\xi_j=c_k$, where $\xi_j(\lambda_m)\neq 0$. It follows that $m+c$ cannot be in $Z(Q_{i,m})$, a contradiction with $m+c\in Z(J)$.

Hence $Z(J)$ is a union of points. Choose any basis $e_1,\ldots,e_r$ of $\mf{t}$. For each $i$ we have $P_i(e_i)\in J$ for some polynomial $P_i$. Hence $D_{P_i}(F)=0$ and $F=\sum S_{ij}e^{\alpha_{ij}e_i^*}$, where $S_{ij}$ is a power series that has finite degree in $e_i^*$. Write $F=\sum S_{1j} e^{\alpha_{1j}e_1^*}$. Since $P_2(\partial_2)(F)=0$ and $P_2(\partial_2)(e^{\alpha_{1j}e_1^*})=0$, we have $P_2(S_{1j})=0$, hence $S_{1j}=\sum_k S_{12k}e^{\alpha_{2k}e_2^*}$. Continuing in this way, we get $F=\sum S_{j_1\cdots j_r}e^{\sum \alpha_{ij_i}e_i^*}$, where $S_{j_1\cdots j_r}$ are polynomials.
\end{proof}
Now $u_T$ is a function defined almost everywhere and not just a power series. We can shift the argument to get a trace for a different automorphism:
\begin{prop}
Let $h$ be an automorphism of $\mc{A}$ that acts on $r^{\lambda}$ as $e^{2\pi\rmi \eta(\lambda)}$ and $T$ be a $g$-twisted trace with a generating function $u_T$. Suppose that $u_T(y-\pi \rmi \eta)$ is regular at zero. Then the Taylor expansion of $u_T(y-\pi\rmi\eta)$ at zero gives an exponential generating function of a $gh$-twisted trace.
\end{prop}
\begin{proof}
The FI parameter for $gh$ is $\zeta+\eta$. 

We will check the trace condition in the form \[(e^{\tfrac12\lambda}-e^{2\pi\rmi(\zeta+\eta)(\lambda)}e^{-\tfrac12\lambda})u_T(y-\pi \rmi \eta)\in\Ker D_{\lambda}\] The operator $D_{\lambda}$ has constant coefficient, its solution remain solutions after the variable shift. Shifting the function by $\pi \rmi\eta$ gives
\[(e^{\tfrac12\lambda}e^{\pi\rmi\eta(\lambda)}-e^{\pi\rmi(2\zeta+\eta)(\lambda)}e^{-\tfrac12\lambda})u_T(y).\] Dividing by $e^{\pi\rmi\eta(\lambda)}$ gives precisely the trace condition for $u_T$.
\end{proof}

\subsection{Traces as a sum of integrals}
\label{SubSecSumOfIntegrals}
\begin{defn}
\begin{enumerate}
\item
An exponential fraction (corresponding to parameter $\zeta$) is a function $f$ regular at zero of the form \[\frac{\sum_{\mu}S_{\mu}e^{\mu}}{\prod_{\lambda\in\mf{t}_{\R}}(e^{\tfrac12\lambda}-e^{2\pi\rmi\zeta(\lambda)}e^{-\tfrac12\lambda})},\] where the sum in  the numerator and the product in the denominator are finite.
\item
Integration weight is a composition of a projection $p\colon \mf{t}^*\to V$ and an exponential fraction $f$ that is exponentially decaying at infinity.
\end{enumerate}
\end{defn}
We note that an exponential fraction $f$ decays exponentially at infinity if and only if all real parts $\Re\mu$ of exponents in the numerator belong to the interior of the convex hull $P_f=\sum [-\tfrac12,\tfrac12]\lambda$ of exponents in the denominator. Note that integration weight $f$ is itself an exponential fraction, the affine hull of $P_f$ is a subspace $U_{\R}=V_{\R}^*$, and each $\mu$ belongs to the relative interior of $P_f$.

In this case we say that $U_{\R}$ is an integration domain of $f$. If $f$ is an integration weight then $\mc{F}f=\delta_{U_{\R}} g$, where $U_{\R}$ is the integration domain of $f$, and $g$ is the Fourier transform of the corresponding exponentially decaying function.
\begin{cor}
\label{CorReducedExponentSum}
We can write $u_T=\sum e^{\zeta_i}Q_if_i$ such that \begin{enumerate}
\item
Each $f_i$ is an integration weight, $\zeta_i\in \mf{t}$, $Q_i\in \C[U^{\perp}]$.
\item
If $f_i$ and $f_j$ have the same integration domain $U$ and $\zeta_i\neq \zeta_j$ then $\zeta_i-\zeta_j\notin \C U$.
\item
There is only one term $e^{\zeta_0}f_0$ with the integration domain $\R^n$ and $\zeta_0$ can be taken arbitrarily small.
\end{enumerate}. 
\end{cor}
\begin{proof}
We start with the first condition. Assume that $u_T$ is not exponentially decaying. Let $x^ae^{\zeta}$ be a term in the numerator the does not belong to $M^{\circ}=\sum_{l\in I}(-\tfrac12\lambda_l,\tfrac12\lambda_l)$. Multiplying $u_T$ by a small exponent if necessary we can write $\zeta=\zeta_0+\sum d_l\lambda_l$, where $\zeta_0$ is in $M^{\circ}$ and $d_l$ are integers.

Then \[\frac{x^ae^{\zeta}-x^ae^{\zeta_0}}{\prod_{l\in I}(e^{\tfrac12\lambda_l}-e^{2\pi\rmi\zeta(\lambda)}e^{-\tfrac12\lambda_l})}\] is a sum of fractions with smaller number of linear terms in the denominator. Continuing in this way we get $u_T=\sum e^{\zeta_i}Q_if_i$, where the exponents in the numerator of each $f_i$ belong to the relative interior of the convex hull of the exponents in the denominator, but $f_i$ are not necessarily regular at zero. Moreover, if the integration domain of $f_i$ is the whole $\R^n$ then $e^{\zeta_i}f_i$ is exponentially decaying. We add all these terms to get one exponentially decaying fraction $f_0$ with integration domain $\R^n$. Hence $e^{\eps}u_T=f_0+\cdots$, where $\cdots$ contains the terms with smaller integration domain. This proves the third condition

We turn to the second condition. If it does not hold, we can assume that there are $i,j$ such that $\zeta_i=\zeta_j+r$. Take such $i,j$ with the largest possible dimension $D$ of the integration domain. We write \[e^{\zeta_i}Q_if_i=e^{\zeta_j}Q_if_i+e^{\zeta_j}Q_i(e^r-1)f_i.\] We can decompose \[(e^r-1)f_i=\sum e^{\zeta'_{ik}}f_{ik}\] with the condition as above: exponentially decaying function plus functions with strictly smaller integration domains. Therefore we rewrite $e^{\zeta_i}Q_if_i$ as $e^{\zeta_j}Q_i(f_i+e^{\eps_0}f_{i0})$ plus terms with smaller integration domains. Here $\eps_0$ can be taken arbitrarily small. Hence $e^{\zeta_i}Q_if_i+e^{\zeta_j}Q_jf_j=e^{\zeta_j+\eps_0}(Q_ie^{-\eps_0}f_i+Q_if_{i0}+Q_je^{-\eps_0}f_j)$ plus terms with smaller integration domains. Since $\eps_0$ can be taken arbitrarily small, we can make both $e^{-\eps_0}f_i$ and $e^{-\eps_0}f_j$ exponentially decaying. After this operation we have one less term with integration domain of dimension $D$, hence this process will end.

We get the sum satisfying both conditions, but where $f_i$ can be singular at zero.

It remains to prove that we can take $f_i$ to be regular at zero. We claim that all $f_i$ are already regular at zero in any expression satisfying both conditions.  Suppose not. Take $f_i$ singular at zero with the largest possible integration domain $U$. Among those, take $f_i$ with $Q_i$ of maximal degree. Take any other singular $f_j$ with integration domain $U_j$. Then for any $\xi$ such that $\xi(U_j)=\{0\}$ for large anough $M_j$ we have $(\partial_{\xi}-\xi(\zeta_j))^{M_j}(e^{\zeta_j}Q_jf_j)=0$. Since $U$ is maximal either we can take $\xi$ that is nonzero on $U$ or, in the case when $U_j=U$ we can take $\xi$ such that $\xi(\zeta_i)\neq \xi(\zeta_j)$. If $U_i=U_j, \zeta_i=\zeta_j$, we take a differential operator in $U_i^{\perp}$ that sends $Q_i$ to $1$, hence sends $Q_j$ to zero. Taking product of these linear operators we get a differential operator $D$ that sends all other singular terms to zero. The function $D(e^{\zeta_i}Q_if_i)$ is still singular: in the formula for the derivative of a product we apply $\partial_{\xi}$ to $f_i$ if $\xi(U)\neq \{0\}$ to get a singular term that does not cancel with any other terms.

Now we see that $D(e^{\zeta_i}Q_if_i)$ plus regular terms equal to $D(u_T)$ which is regular. We get a contradiction with the fact that $D(e^{\zeta_i}Q_if_i)$ is singular.

\end{proof}
\begin{rem}
On the Fourier dual side the second condition essentially means that we can shift contours in the complex subspace $\C U$ to get integral over $U$ plus integrals over smaller subspaces corresponding to poles of the meromorphic integration weight.
\end{rem}
Taking Fourier dual, we get the following corollary:
\begin{cor}
\label{CorTraceIsAnIntegral}
Any trace $T$ can be written in the form \[T(R)=\sum_{U,D,\zeta}\int_{U+i\zeta}(DR)(ix)w_{U,D,\zeta}(x-i\zeta)dx,\] where $U$ is a real subspace, $D$ is a differential operator and $w_{U,D,\zeta}$ is an exponentially decaying function holomorphic on $U\times B_{\eps}(0)\in U_{\C}$.
\end{cor}
\begin{rem}
Classifying traces in this form is complicated because the combinatorics of hyperplane arrangements comes up. For example, even if all subspaces consist of a single point, it is not straightforward to describe all finite subsets $S$ of $\mf{t}_{\C}$ and differential operators $D_s,s\in S$ such that $T(R)=\sum_{s\in S}(DR)(s)$ is a trace.
\end{rem}

We will need more precise information about Fourier transform of exponentially decaying exponential fraction below.
\begin{lem}
\label{LemFourierTransformOfExponentialFraction}
Fourier transform of exponentially decaying exponential fraction is a sum of exponentially decaying exponential fractions.
\end{lem}
\begin{proof}
Let $w=\frac{\sum S_{\chi}e^{\chi}}{\prod (e^{\mu}-e^{4\pi\rmi\zeta(\mu)}e^{-\mu}}$ be an exponential fraction. Let \[w_0=\frac{1}{\prod_{\mu\in S} (e^{\mu}-e^{4\pi\rmi\zeta(\mu)}e^{-\mu})}.\] By Proposition~\ref{PropTraceAsAnIntegral} the function $w_0$ defines a non-twisted trace $T_0$ on the abelian Coulomb branch with weights $\mu\in S$ and flavor parameters coming from $\zeta$. In this case, the exponential generating function $u_0$ is the expansion of $\mc{F}w_0$ at zero. Using Theorem~\ref{ThrFourierTransformIsMeromorphic} we get that $\mc{F}w_0$ is exponential fraction. Since $u_0$ defines a function in $L^2$, this exponential fraction is exponentially decaying

Since $e^{\chi}w_0$ is still exponentially decaying, we have $\mc{F}(e^{\chi}w_0)=(\mc{F}w_0)(y+\rmi\chi)$. This is still exponential fraction, it is in $L^2$, hence exponentially decaying  Multiplying $e^{\chi}w_0$ by $S_{\chi}$ applies a differential operator on Fourier dual side, hence the result is still exponentially decaying exponential fraction. It remains to take the sum over all $\chi$ that appear in the numerator.
\end{proof}

\section{Positive traces}
\label{SecPositive}
\subsection{Exponential generating function of a positive trace}
\label{SubSecGeneratingFunctionBounded}

In this subsection we show that any trace that satisfies $T(R(z)\ovl{R}(-z))\geq 0$ for all $R\in\C[\mf{t}]$ can be written in the form $T(R(z))=\sum_{U\subset\mf{t}_{\R}}R(ix)w_U(x)dx$, where the sum is taken over a finite number of real subspaces $U$ of $\mf{t}_{\R}$ and $w_U$ is an exponentially decaying function holomorphic on $U\times B_{\eps}(0)\in U_{\C}$. The important difference with Corollary~\ref{CorTraceIsAnIntegral} is that all subspaces are in $\mf{t}_{\R}$ and there are no differential operators. We will show that each of the terms of the sum on its own defines a trace.

If we restrict $u$ to a generic line, each term $e^{\tfrac12\lambda_i}-s_i e^{-\tfrac12\lambda_i}$ in the denominator becomes a nonzero function $e^{\beta_i t}-s_ie^{-\beta_i t}$.
\begin{lem}
Let $g(x)$ be a function of the form
\begin{equation}
\label{EqOneDimExponentOverExponent}
g(x)=\frac{\sum_{i\in I} e^{\alpha_i x}P_i(x)}{\prod_{j\in J}(e^{\beta_j x}-s_je^{-\beta_j x})}
\end{equation} regular at zero such that \[\bigg(P(\frac{d}{dx})\ovl{P}(-\frac{d}{dx})u\bigg)(0)>0\] for all nonzero polynomials $P$. Then $\abs{g(x)}\leq g(0)$ for all $x\in \RN$.
\end{lem}
\begin{proof}
Dividing with residue we write 
\begin{equation}
\label{EqGisG0PlusExponents}
 g(x)=g_0(x)+\sum_k R_k(x)e^{\gamma_k x},
\end{equation}
where

\[g_0(x)=\frac{\sum_{i\in I_1} e^{\alpha'_i x}Q_i(x)}{\prod_{j\in J}(e^{\beta_j x}-s_je^{-\beta_j x})}\] and satisfies $\abs{\alpha'_i}\leq \sum_j \beta_j$ for all $i$; $R_k$ are polynomials and $\Re\gamma_k\neq 0$. First, we prove that all $R_k$ must be zero. Assume not.

Looking at the growth of $g_0$ at plus and minus infinity we get $g_0=g_s+S_0(x)+S_1(x)\tanh(x)$, where $g_s$ is exponentially decaying and $S_0, S_1$ are polynomials. Let $D=A(\frac{d}{dx})$ be a differential operator such that $D(S_0+S_1)=D(S_0-S_1)=0$. Then $D(g_0)$ is an exponentially decaying function. Now, the function $A(\frac{d}{dx})\ovl{A}(-\frac{d}{dx})g(x)$ still satisfies the assumptions of the lemma, it has decomposition as in~\eqref{EqGisG0PlusExponents} with all $R_k$ nonzero, but now $g_0$ is exponentially decaying. Abusing notation, denote $A(\frac{d}{dx})\ovl{A}(-\frac{d}{dx})g(x)$  by $g$.  

Let $T(R(z))=\big(R(\frac{d}{dx})g\big)(0)$. Then $T(R(z))=\int_{\R}R(iz)(\mc{F}g_0)(z)dz+\Phi(R)$, where $\Phi$ is a linear function of the form $\Phi(R)=\sum (D_j R)(z_j)$, where $D_j$ are differential operators and $\Re z_j\neq 0$. Let $U(z)=\ovl{U}(-z)$ be a polynomial that has roots at $z_j$ and $-\ovl{z_j}$.

 It remains to use the proof of Proposition 4.17 in~\cite{EKRS} to get a contradiction. Namely, using Lemma~4.18 from~\cite{EKRS} we find a polynomial $F$ such that $\Phi(F(z)\ovl{F}(-z))<0$. Using Lemma~4.2 from~\cite{EKRS} we find a sequence of polynomials $S_n$ such that $U(iz)S_n(iz)$ tends to $F(iz)$ in $L^2(\RN,\abs{\mc{F}g_0})$. It follows that $T\big ((F-US_n)(\ovl{F}(-z)-U\ovl{S_n}(-z))\big)$ tends to $\Phi_g(F(z)\ovl{F}(-z))<0$, a contradiction.

Hence $g=g_0(x)$ is bounded. As above, we have $g=g_s+S_0(x)+S_1(x)\tanh(x)$ and there exists a differential operator $D=A(\frac{d}{dx})$ such that $h=A(\frac{d}{dx})\ovl{A}(-\frac{d}{dx})g$ is exponentially decaying. Let $T_h(R(z))=\big(R(\frac{d}{dx})h\big)(0)$. Then $T_h(R(z))=\int_{\R}R(iz)w(z)dz$, where $w$ is the Fourier transform of $h$. It follows that $T(R(z))=\int_{\R}(R-R_0)(iz)w(z)dz+\Psi(R_0(z))$, where $R_0$ is the remainder from division of $R$ by $A(z)\ovl{A}(-z)$ and $\Psi$ is some linear function on the space of polynomial of degree less than $2\deg P_0$. Then $T$ has the same form as the functional in the Subsection~4.3 of~\cite{EKRS}. The proofs work without change and give $T(R(z))=\int_{\R}R(iz)w(z)dz+\sum_j a_j R(c_j)$, where $w(z)$ is a nonnegative exponentially decaying function, $c_j$ are imaginary and $a_j$ are positive. It follows that $g=\mc{F}w+\sum a_j e^{c_j x}$. Because $w$ is nonnegative, we have $\abs{\mc{F}w (x)}\leq \mc{F}w(0)$ for real $x$. Since $c_j$ are imaginary, we have $\abs{e^{c_j x}}=1$ for real $x$. It follows that $\abs{g(x)}\leq g(0)$.

\end{proof}

\begin{cor}
\label{CorNumeratorOfU}
Let $T$ be a positive trace, $u$ be its exponential generating function. Then $u$ is bounded on $\mf{t}_{\R}^*$. As a corollary, for any $R_{\lambda}e^{\lambda}$ in the numerator, $\lambda$ belongs to $\mc{P}=\sum_{l\in I}[-\tfrac12,\tfrac12]\lambda_l$. Let $R_{\lambda}e^{\lambda}$ be a term such that $\lambda$ belongs to the face $F$ of $\mc{P}$. Suppose that $U$ is a linear part of the affine hull of $F$. Then $R_{\lambda}$ belongs to $\C[U_{\C}^*]$.
\end{cor}
\begin{proof}

Using lemma we get that $\abs{u(x)}\leq u(0)$ for generic $x$, hence for all $x$. 

If the second statement is not true, we can find $\phi\in \mf{t}_{\R}^*$ such that $\phi(\lambda)$ is greater than the value of $\phi$ on $\mc{P}$ or such that $\phi(\lambda)$ is the largest element in $\phi(\mc{P})$, but $R_{\lambda}|_{\C\phi}$ is not constant. Both cases contradict $u$ being bounded.
\end{proof}
\begin{prop}
\label{PropReducedSumForBoundedU}
Let $u$ be a bounded exponential fraction. Then in all of the terms in the sum of Corollary~\ref{CorReducedExponentSum} we can take $Q_i=1$ and $\Re\zeta_i$ arbitrarily small.
\end{prop}
\begin{proof}
Fix small $\eps\in \mf{t}$. We are reducing function $e^{\eps}u$. The only nontrivial reduction steps happen when we consider a boundary term $e^{\rmi\chi}\frac{e^{\lambda}R_{\lambda}}{\prod_{i}(e^{\mu_i}-e^{-\mu_i})}$ with $\lambda\in F$ and $R_{\lambda}\in \C[U_{\C}^*]$. Denote by $I$ the set of indices $i$ such that $\mu_i\notin U$. Then $\lambda=\sum a_i\mu_i$ with $a_i=\pm 1$ for $i\in I$ and $\abs{a_i}<1$ for $i\notin I$.

Let $\eps=\sum b_i\mu_i$ with all $b_i$ small. Let $J\subset I$ consist of all $i$ such that $a_i$ and $b_i$ have the same sign. We note that $\lambda_0=\lambda-2\sum_{i\in J}a_i\mu_i+\eps$ belongs to the interior of $\mc{P}$. Let $J=\{j_1,\ldots,j_k\}$. This gives the required reduction step: \[\frac{e^{\lambda}R_{\lambda}}{\prod_{i}(e^{\mu_i}-e^{-\mu_i})}=\frac{e^{\lambda_0}R_{\lambda}}{\prod_i (e^{\mu_i}-e^{-\mu_i})}+\sum_{l=1}^k \frac{e^{\lambda-\sum_{j=1}^{l}a_j\mu_j}R_{\lambda}}{\prod_{i\notin \{j_1,\ldots,j_l\}}(e^{\mu_i}-e^{-\mu_i})}.\] 
  Note that each term satisfies the same assumptions because $\lambda-\sum_{j=1}^l a_j\mu_j=\sum a_i\mu_i$ with $a_i=\pm 1$ for $i\in I\setminus\{j_1,\ldots,j_k\}$ and the face $F$ stays the same. Continuing in this way we get a decomposition $u=\sum v_i$, where all terms $v_i$ are bounded and $e^{\eps}v_i$ is an integration weight.
\end{proof}
\begin{cor}
Let $T$ be a positive trace. Then $T(R(z))=\int_{U\subset \mf{t}_{\R}} R(\rmi z)w_U(z)dz$, where the sum is taken over a finite number of affine subspaces $U$ of $\mf{t}_{\R}$, functions $w_U$ are exponentially decaying with at most simple poles on $U$. If poles are present, the integral means $\lim_{\eps\to 0}\int_{U+\rmi\eps}R(\rmi z)w_U(z)dz$, where the limit is taken in some particular direction.
\end{cor}
\begin{proof}
When $T$ is positive, $u$ is bounded and we use the previous corollary. Then we take Fourier transform of $u$ in generalized sense, similar to the proof of Corollary~\ref{CorTraceIsAnIntegral}.

Suppose that $\mc{F}u$ has a pole at $\xi-s$ of order $k\geq 2$, let $g(\xi-s)^{-k}$ be the leading term. Since $g$ is nonzero, $\mc{F}g(\mu)\neq 0$ for some real $\mu\in \xi^{\perp}$. Shifting $u$ by $\mu$ multiplies $\mc{F}u$ by $e^{\rmi\mu}$ and we get $\int_{\xi^{\perp}}g(z)dz\neq 0$. Then $\int_{\xi^{\perp}}\mc{F}u dz$ is a function of a single variable that has a pole of order $k$. Inverse Fourier transform of this function is $u_1=u|_{\C\xi}$. We can write $u_1=a+b\tanh(x)+u_2$, where $u_2$ exponentially decays at infinity. Then $\mc{F}u_1$ has at most first order pole, a contradiction.

\end{proof}
\subsection{Positivity in terms of weights}
\label{SubSecAnswer}
We will use the following lemma, its proof is completely analogous to the proof of Lemmas~4.1 and~4.2 in~\cite{EKRS}.
\begin{lem}
\label{LemApproximation}
Let $\mu=\sum w_i \delta_{V_i}$ be a measure, where $V_i$ are affine subspaces of $\R^d$ and $w_i$ are exponentially decaying at infinity. Then 
\begin{enumerate}
\item
Polynomials are dense in $L^2(\R^d,\mu)$ 
\item
Moreover, $\C[z_1,\ldots,z_d]H$ is dense in $L^2(R^d,\mu)$ for any polynomial $H$ that does not vanish on any of $V_i$.
\item
The closure of the span of $\{R(z)\ovl{R}(z)\mid R\in \C[z_1,\ldots,z_d]\}$ is the subset of all nonnegative functions.
\end{enumerate}
\end{lem}

\begin{prop}
\label{PropWeightsAreNonnegativeAndRegular}
Let $T(R)=\sum_U\int_{U}R(\rmi z)w_U(z)dz$. Then each $w_U$ is nonnegative and regular on $U$.
\end{prop}
\begin{proof}
If a function is nonnegative and has at most simple poles, it is regular.

We prove nonnegativity by downward induction on $\dim U$. The base case is $\dim U>d$, which is vacuous. We turn to the induction step.

Let $R_0(\rmi z)$ be a real polynomial that cancels out the poles of $w_U$ and that vanishes on all subspaces $V$ that do not contain $U$. Suppose that $w_U$ is not nonnegative. Then by Lemma~\ref{LemApproximation} applied to $R_0(\rmi z)^2w_U(z)$ there exists $R_1$ such that \[\int_U R_1(\rmi z)\ovl{R_1}(-\rmi z)R_0(\rmi z)^2 w_U(z)dz<0.\]

For any $R$ we have \[T(R(z)\ovl{R}(-z)R_0(z)^2)=\int_U R(\rmi z)\ovl{R(\rmi z)}R_0(\rmi z)^2w_U(z)dz+\int_{\R^d}R(\rmi z)\ovl{R(\rmi z)} d\mu,\] where $\mu$ is the sum of measures corresponding to subspaces containing $U$, multiplied by $R_0(\rmi z)^2$.

Let $S$ be a polynomial that vanishes on $U$, but does not vanish on subspaces containing $U$. Then there exists a sequence $Q_n$ such that $R_n(iz)=S(iz)Q_n(iz)-R_1(\rmi z)$ tends to zero in $L^2(\R^d,\mu)$. Then
\[T(R_n(z)\ovl{R_n}(-z)R_0(z)^2)=\int_U R_1(\rmi z)\ovl{R_1(\rmi z)}R_0(\rmi z)^2w_u(z)dz+\int_{\R^d}R_n(iz)\ovl{R_n(iz)}d\mu.\] The second term tends to zero, the first is negative, hence for some $n$ we get a contradiction with positivity.
\end{proof}
\begin{prop}
\label{PropEachTermIsATrace}
Let $T(R)=\sum_{U}\int_{U}R(\rmi z)w_U(z)dz$ be a trace, where $w_U$ are exponentially decaying exponential fractions. Then each term defines a trace on $\mc{A}$. More precisely, the term corresponding to $U\subset \mf{t}_{\R}$ is a composition of a homomorphism of algebras $\phi\colon \mc{A}\to \mc{A}_U$ to a smaller abelian Coulomb branch and a trace $T_U\colon \mc{A}_U\to \C$ constructed as in Proposition~\ref{PropTraceAsAnIntegral}. In particular, only $\xi_i$ with flavors $b_i$ satisfying $\abs{\Re b_i}<\frac12$ can be used in the construction of the weight function.
\end{prop}
\begin{proof}
We do this by downward induction on the dimension of $U$. The induction base is $\dim U>d$, which is vacuous.

We turn to the induction step. Because of the assumption, we can assume that $U$ is the subspace of maximal dimension such that $w_U\neq 0$. For each $V$ choose $\alpha_V\in\mf{t}^*_{\R}\oplus\R$ such that $\alpha_V(V)=\{0\}$ but $\alpha_V(U)\neq \{0\}$ if $V\neq U$. Shifting the generators in $\mf{t}^*_{\R}\subset \mc{A}$ if necessary, we can assume that $U$ is a linear subspace of $\mf{t}_{\R}$

The trace condition for a fixed $\lambda$ is \[T\big(R(x+\frac{\lambda}{2})P_{\lambda}(x+\frac{\lambda}{2})\big)=e^{2\pi\rmi\zeta(\lambda)}T\big(R(x-\frac{\lambda}{2})P_{\lambda}(x-\frac{\lambda}{2})\big).\]

Suppose that $\lambda\notin U$. If necessary, we can change $\alpha_U$ so that $\alpha_U(\lambda)\neq 0$. Let \[R_0=\prod_{V\neq 0}(\alpha_V\pm \alpha_V(\frac{\lambda}{2}))\cdot (\alpha_U+\alpha_U(\frac{\lambda}{2})).\] Let $R=QR_0$. Then the trace condition becomes
\[\int_U Q(\rmi z+\frac{\lambda}{2})R_0(\rmi z+\frac{\lambda}{2})P_{\lambda}(\rmi z+\frac{\lambda}{2})w_U(z)dz=0.\]
This is possible for all polynomials $Q$ if and only if 
\[R_0(\rmi z+\frac{\lambda}{2})P_{\lambda}(\rmi z+\frac{\lambda}{2})w_U(z)\] vanishes on $U$. Since $R_0(\rmi z+\frac{\lambda}{2})$ and $w_U$ are not identically zero on $U$, we have $P_{\lambda}(\rmi z+\frac{\lambda}{2})$ vanishes on $U$ or, equivalently, on $U_{\C}$. We have \[P_{\lambda}(z+\frac{\lambda}{2})=\prod_{i=1}^n\prod_{j=1}^{-\xi_i(\lambda)}(\xi_i+j-\tfrac12+b_i).\] Here we assume $\xi_i(\lambda)\leq 0$, the case $\xi_i(\lambda)\geq 0$ is similar. One of the terms should vanish on $U$. Hence there exists $i$ such that $\xi_i$ vanishes on $U$ and there exists $1\leq j\leq -\xi_i(\lambda)$ such that $j-\tfrac12+b_i=0$. Equivalently, $\tfrac12-b_i$ is a positive integer at most $-\xi_i(\lambda)$.

It follows that for any $\lambda\notin U$ there exists $\xi_i(U)=\{0\}$ such that $1\leq \tfrac12-b_i\leq -\xi_i(\lambda)$ is an integer. (In particular, such $\xi_i$ span $U^{\perp}$). Let $\mc{A}_U$ be the abelian Coulomb branch corresponding to $\xi_i+U^{\perp}$.  Consider the linear map $\phi\colon \mc{A}\to\mc{A}_U$ that sends $r^{\lambda}R(z)$ to zero if $\lambda\notin U$ and to $r^{\lambda}R(z)|_U$ else. We claim that this map is a homomorphism of algebras. It is enough to check that $r^{\lambda}R(z),\lambda\notin U$ span an ideal in $\mc{A}$. Suppose that $\lambda+\mu\in U$, $\lambda\notin U$. Then 
\[r^{\lambda}r^{\mu}=\prod_{i=1}^n \prod_{j=1}^{d(\xi_i(\lambda),\xi_i(\mu)-1}(\xi_i+b_i+\xi_i(\lambda)+j-\tfrac12).\] Let $i$ be the index such that $1\leq \tfrac12-b_i\leq -\xi_i(\lambda)-1$. Since $\lambda+\mu\in U$, we have $\xi_i(\lambda)=-\xi_i(\mu)$, hence $d(\xi_i(\lambda),\xi_i(\mu))=-\xi_i(\lambda)$, and the $i$-th term in the product vanishes.

Denote by $T_U\colon \C[U_{\C}]\to \C$ the map that sends $R$ to $\int_U R(\rmi z)w_U(z)dz$. We will show that $T_U$ defines a trace (also denoted by $T_U$) on $\mc{A}_U$, hence $T=T_U\circ \phi$ defines a trace on $\mc{A}$.  

In order to do this we consider the case $\lambda\in U$. Take \[R_0=\prod_{V\neq 0}(\alpha_V\pm \alpha_V(\frac{\lambda}{2}))\]

 Let $R=QR_0$. The trace condition gives
\begin{align*}
\int_U Q(\rmi z+ \frac{\lambda}{2})R_0(\rmi z+\frac{\lambda}{2})P_{\lambda}(\rmi z+\frac{\lambda}{2})w_U(z)dz=\\
e^{2\pi\rmi\zeta(\lambda)}\int_U Q(\rmi z- \frac{\lambda}{2})R_0(\rmi z-\frac{\lambda}{2})P_{\lambda}(\rmi z-\frac{\lambda}{2}) w_U(z)dz.
\end{align*} Since $\lambda\in U$, any $\xi_i+c$ with $\xi_i(\lambda)\neq 0$ will not vanish on $U$, so $P_{\lambda}(z\pm\rmi\frac{\lambda}{2})$ is nonzero on $U$.

We get
\begin{align*}
\int_{U-\rmi\frac{\lambda}{2}}Q(\rmi z)R_0(\rmi z)P(\rmi z)w_U(z+\rmi\frac{\lambda}{2})dz=\\
e^{2\pi\rmi\zeta(\lambda)}\int_{U+\rmi\frac{\lambda}{2}}Q(\rmi z)R_0(\rmi z)P(\rmi z)w_U(z-\rmi\frac{\lambda}{2})dz.
\end{align*}

Combining the proof of Proposition~\ref{PropReducedSumForBoundedU} with Lemma~\ref{LemFourierTransformOfExponentialFraction} we get that $w_U$ are exponential fractions. Take $Q$ that cancels all of the poles of $w_U$ on the hyperplanes intersecting $U+\rmi[-\frac{\lambda}{2},\frac{\lambda}{2}]$. Then we can shift the contour and use Lemma~\ref{LemApproximation} to get that $w_U(z+\rmi\frac{\lambda}{2})=e^{2\pi\rmi\zeta(\lambda)}w_U(z-\rmi\frac{\lambda}{2})$. 

It remains to check that $w_U$ has poles at prescribed hyperplanes. 

First, we claim that $R_0(\rmi z)P(\rmi z)w_U(z+\rmi\frac{\lambda}{2})$ is holomorphic on a neighborhood of $U\times \rmi[-\frac{\lambda}{2},\frac{\lambda}{2}]$. Suppose not. Take $Q=Q_0Q_1$, where $Q_0$ cancels all poles of $R_0(\rmi z)P(\rmi z)w_U(z+\rmi\frac{\lambda}{2})$ along the hyperplanes except one. Choose a decomposition $U=W\oplus\R\lambda$. Applying the residue theorem on $(R\times [-\tfrac12,\tfrac12])\lambda+z_0$, for all $z_0\in W$ we get $\int_{W}Q_0(z)Q_1(z)r_W(z)dz=0$, where $r_W$ is the residue. Since $r_w$ is $(e^{\chi}-se^{-\chi})w_U$ restricted to a hyperplane $\chi=const$ for some $\chi,s$, it is also an exponential fraction that is exponentially decaying. Using Lemma~\ref{LemApproximation} we get a contradiction.

Hence $R_0(\rmi z)P(\rmi z)w_U(z+\rmi\frac{\lambda}{2})$ is holomorphic on a neighborhood of $U\times \rmi[-\frac{\lambda}{2},\frac{\lambda}{2}]$. Consider a term $\alpha_V\pm \alpha_V(\frac{\lambda}{2})$ of $R_0$. When $\Im z\in (-\frac{\lambda}{2},\frac{\lambda}{2})$, we have $\Re(\alpha_V(\rmi z))\in (-\frac{\alpha_V(\lambda)}{2},\frac{\alpha_V(\lambda)}{2})$. Hence $R_0$ does not vanish at any point of $U\times \rmi(-\frac{\lambda}{2},\frac{\lambda}{2})$. It follows that $P(\rmi z)w_U(z+\rmi\frac{\lambda}{2})$ is holomorphic on a neighborhood of $U\times \rmi(-\frac{\lambda}{2},\frac{\lambda}{2})$. 

Suppose that $w_U$ has a pole of the form $\xi+\rmi s$. Since the only terms appearing in $P_{\lambda}$ have form $\xi_i+c$, we have $\xi=\xi_i$ for some $i$. Moreover, we can take $\Re s\in [-\tfrac12,\tfrac12]$, $s\neq 0$. Also, if $\Re s<0$ we change all $\xi_j=\xi_i$ to its opposite, making $\Re s>0$.  We claim that the order of this pole is not higher than the number of $b_j$ with $\xi_j=\xi_i$ and $b_j=s-\tfrac12$.

Suppose not. Take $\lambda$ such that $\xi_i(\lambda)=1$. For generic such $\lambda$, the intersection of hyperplanes $\xi_i+a$ and $\xi_j+b$ for $\xi_j\neq \xi_i$ with $U\times \rmi(-\frac{\lambda}{2},\frac{\lambda}{2})$ are different. Hence, the only terms of $P_{\lambda}$ that could cancel the poles of $w_U$ on $U\times \rmi(-\frac{\lambda}{2},\frac{\lambda}{2})$ come from $\xi_j=\xi_i$. 

This term is $\rmi\xi_j+b_j$. On the other hand, if $w_U$ has a pole of order $N$ at $\xi_i+s+\rmi\Z$, then $w_U(z+\rmi\frac{\lambda}{2})$ has a pole of order $N$ at $\xi+\rmi s-\tfrac12\rmi+\rmi\Z$. It follows that there should be at least $N$ terms with $-b_j=-s+\tfrac12$, hence $b_j=s-\tfrac12$. 

This means that $\prod_{j}(e^{2\pi\xi_j}+e^{2\pi\rmi b_j})w_U(z)$ is entire, as in Proposition~\ref{PropTraceAsAnIntegral}.
\end{proof}
\begin{prop}
\label{PropEachTermIsAPositiveTrace}
Let $T(R)=\sum_U \int_U R(\rmi z)w_U(z)dz$ be a trace. Then $T$ is positive if and only if each $w_U\neq 0$ gives a positive trace on $\mc{A}_U$.
\end{prop}
\begin{proof}
Suppose not. This means that $w_U$ has a value with a wrong sign on $U+\rmi\frac{\lambda}{2}$ for some $\lambda$. Using Lemma~\ref{LemApproximation} we get that for any $R_0\in \C[\mf{t}]$ that does not vanish on $U$ there exists $R\in \C[\mf{t}]$ such that for $a=r^{\lambda}R(z)R_0(z)$ we have $T_U(a\rho(a))<0$. We take $R_0$ so that $T_V(a\rho(a))=0$ for any $a=r^{\lambda}R(z)R_0(z)$ and any $V$ that does not contain $U$. Fix $R_1$ such that $T_U(a\rho(a))<0$. Let $R_2$ be a polynomial such that for $b=a+r^{\lambda}R_2(z)S(z)$ we have $T_U(b\rho(b))=T_U(a\rho(a))$. Using Lemma~\ref{LemApproximation} we find a sequence $S_n(z)$ such that $b$ tends to zero in the appropriate $L^2$ space. It follows that $T(b\rho(b))$ tends to $T_U(a\rho(a))<0$, contradiction with positivity of $T$.
\end{proof}

It remains to combine the results of Propositions~\ref{PropEachTermIsAPositiveTrace},~\ref{PropEachTermIsATrace} and~\ref{PropPositivityOnTermsOfWeight}:
\begin{thr}
\label{ThrAnswer}
Let $\mc{A}$ be an abelian Coulomb branch corresponding to weights $\xi_1,\ldots,\xi_n$ and flavor parameters $b_1,\ldots,b_n$ and $\rho$ be the antilinear automorphism of $\mc{A}$ corresponding to $\zeta\in\mf{t}_{\R}^*$. 

Let $F$ be the set of affine subspaces $U+\chi$ of $\mf{t}_{\R}$ such that there exists a homomorphism of algebras $\mc{A}\to\mc{A}_U$ that sends  $r^{\lambda}R(z)$ to zero if $\lambda\notin U$ and to $r^{\lambda}R(z+\rmi\zeta)|_U$ else. Here $\mc{A}_U$ is the abelian Coulomb branch corresponding to weights $\xi_1^U=\xi_1+U^{\perp},\ldots,\xi_n^U=\xi_n+U^{\perp}$ and the same flavors. For each $\mc{A}_U$ we take the symmetrized fundamental polytope for ``good'' weights \[\mc{P}_U=\sum_{\abs{\Re b_i}<\tfrac12} [-\tfrac12,\tfrac12]\xi_i^U,\] intersect $\mc{P}_U^{\circ}+\zeta$ with $(2\Z)^{\dim U}\subset U^*$ and take convex hull. Denote the result by $\mc{P}_U^+$. 

If $\mc{P}_{\mf{t}_{\R}}^+$ is empty, there are no positive traces. Else, the cone of positive traces has dimension equal to the sum over $U+\zeta\in F$ of number of integer points in $\mc{P}_U^+$ (in the case of $U=\{0\}$ we add one to the sum).
\end{thr}

\end{document}